\documentclass[a4,12pt]{amsart}
\usepackage{amssymb,amstext,amsmath,amscd,amsthm,amsfonts,enumerate,graphicx,latexsym,color}
\usepackage[all]{xy}
\newtheorem{thm}{Theorem}[section]
\newtheorem*{thm*}{Theorem}
\newtheorem{tthm}{Theorem}

\newtheorem{cor}[thm]{Corollary}
\newtheorem{lem}[thm]{Lemma}

\theoremstyle{definition}
\newtheorem{dfn}[thm]{Definition}
\newtheorem*{dfn*}{Definition}
\newtheorem{rem}[thm]{Remark}
\newtheorem{ques}[thm]{Question}

\newtheorem{ex}[thm]{Example}

\newtheorem{conv}[thm]{Convention}
\theoremstyle{remark}
\newtheorem*{ac}{Acknowledgments}
\newtheorem{claim}{Claim}
\newtheorem*{claim*}{Claim}
\renewcommand{\qedsymbol}{$\blacksquare$}
\numberwithin{equation}{thm}
\def\Ext{\operatorname{\mathsf{Ext}}}
\def\grade{\operatorname{\mathsf{grade}}}
\def\syz{\mathsf{\Omega}}
\def\mod{\operatorname{\mathsf{mod}}}
\def\op{\mathsf{op}}
\def\s{\mathsf{S}}
\def\height{\operatorname{\mathsf{ht}}}
\def\spec{\operatorname{\mathsf{Spec}}}
\def\depth{\operatorname{\mathsf{depth}}}
\def\X{\mathcal{X}}

\def\m{\mathfrak{m}}
\def\p{\mathfrak{p}}
\def\q{\mathfrak{q}}

\def\r{\mathfrak{r}}
\def\ss{\mathfrak{s}}

\def\dim{\operatorname{\mathsf{dim}}}
\def\Hom{\operatorname{\mathsf{Hom}}}
\def\tr{\mathsf{Tr}}
\def\lhom{\operatorname{\mathsf{\underline{Hom}}}}
\def\lmod{\operatorname{\mathsf{\underline{mod}}}}
\def\H{\operatorname{\mathsf{H}}}
\def\one{\mathbf{1}}
\def\R{T}
\def\L{S}
\def\proj{\operatorname{\mathsf{proj}}}
\def\d{\mathsf{D}}
\def\ann{\operatorname{\mathsf{ann}}}
\def\ass{\operatorname{\mathsf{Ass}}}
\def\assh{\operatorname{\mathsf{Assh}}}
\def\v{\mathsf{V}}
\def\C{\mathcal{C}}
\def\Y{\mathcal{Y}}
\def\D{\mathcal{D}}
\def\Im{\operatorname{\mathsf{Im}}}
\def\Min{\operatorname{\mathsf{Min}}}
\begin{document}
\setlength{\baselineskip}{14.84pt}
\title[Extension closed syzygies and locally Gorenstein rings]{Extension closedness of syzygies and\\
local Gorensteinness of commutative rings}
\author{Shiro Goto}
\address{Department of Mathematics, School of Science and Technology, Meiji University, 1-1-1 Higashi-mita, Tama-ku, Kawasaki 214-8571, Japan}
\email{goto@math.meiji.ac.jp}
\author{Ryo Takahashi}
\address{Graduate School of Mathematics, Nagoya University, Furocho, Chikusaku, Nagoya, 464-8602, Japan}
\email{takahashi@math.nagoya-u.ac.jp}
\urladdr{http://www.math.nagoya-u.ac.jp/~takahashi/}
\thanks{The first author was partly supported by JSPS Grant-in-Aid for Scientific Research (C) 25400051. The second author was partly supported by JSPS Grant-in-Aid for Scientific Research (C) 25400038}
\begin{abstract}
We refine a well-known theorem of Auslander and Reiten about the extension closedness of $n$th syzygies over noether algebras.
Applying it, we obtain the converse of a celebrated theorem of Evans and Griffith on Serre's condition $(\s_n)$ and the local Gorensteiness of a commutative ring in height less than $n$.
This especially extends a recent result of Araya and Iima concerning a Cohen--Macaulay local ring with canonical module to an arbitrary local ring.
\end{abstract}
\maketitle
\section{Introduction}\label{sect1}

In this paper we are interested in the following theorem.

\begin{thm}[Evans--Griffith, Araya--Iima]\label{egai}
Let $n\ge0$ be an integer.
Let $R$ be a commutative noetherian ring satisfying Serre's condition $(\s_n)$.
Consider the following conditions.
\begin{enumerate}[\rm(1)]
\item
One has $\syz^n(\mod R)=\s_n(R)$.
\item
The local ring $R_\p$ is Gorenstein for all $\p\in\spec R$ with $\height\p<n$.
\end{enumerate}
Then the implication {\rm(2)} $\Rightarrow$ {\rm(1)} holds.
The opposite implication {\rm(1)} $\Rightarrow$ {\rm(2)} holds if $R$ is a Cohen--Macaulay local ring with canonical module.
\end{thm}

Let us explain some notation.
For a right noetherian ring $\Lambda$, we denote by $\mod\Lambda$ the category of finitely generated right $\Lambda$-modules, and by $\syz^n(\mod\Lambda)$ the full subcategory of $n$th syzygies.
For a commutative noetherian ring $R$, we denote by $\s_n(R)$ the full subcategory of $\mod R$ consisting of modules satisfying Serre's condition $(\s_n)$.
The first assertion of the theorem is a celebrated result of Evans and Griffith \cite[Theorem 3.8]{EG}, while the second assertion has recently been shown by Araya and Iima \cite[Theorem 2.2]{AI}.

It is natural to ask whether the implication (1) $\Rightarrow$ (2) in Theorem \ref{egai} holds for arbitrary commutative noetherian rings, and the main purpose of this paper is to give an answer to this question.
Since $\s_n(R)$ is an extension closed subcategory of $\mod R$, the equality $\syz^n(\mod R)=\s_n(R)$ implies the extension closedness of $\syz^n(\mod R)$.
We thus study when the subcategory $\syz^n(\mod R)$ is extension closed, rather than when one has $\syz^n(\mod R)=\s_n(R)$.
The extension closedness of syzygies has been investigated over a noether algebra by Auslander and Reiten \cite[Theorem 0.1]{AR}.

\begin{thm}[Auslander--Reiten]\label{ar}
Let $\Lambda$ be a noether algebra.
Then the following are equivalent for each nonnegative integer $n$.
\begin{enumerate}[\rm(1)]
\item
$\syz^i(\mod\Lambda)$ is extension closed for all $1\le i\le n$.
\item
$\grade_\Lambda\Ext_{\Lambda^\op}^i(M,\Lambda)\ge i$ for all $1\le i\le n$ and $M\in\mod\Lambda^\op$.
\end{enumerate}
\end{thm}

The first condition of this theorem is too strong for our purpose in that it requires the extension closedness of $i$th syzygies for {\em all} integers $i$ with $1\le i\le n$, while it is the extension closedness of $n$th syzygies that we want to deal with.

Our first main result is the following theorem on not-necessarily-commutative rings.
This in fact provides a refinement of the implication (1) $\Rightarrow$ (2) in Theorem \ref{ar}.

\begin{tthm}\label{im1}
Let $\Lambda$ be a noether algebra such that $\syz^n(\mod\Lambda)$ is extension closed.
Let $M$ be a finitely generated $\Lambda^\op$-module with $\grade_\Lambda\Ext_{\Lambda^\op}^i(M,\Lambda)\ge i-1$ for all $1\le i\le n$.
Then $\grade_\Lambda\Ext_{\Lambda^\op}^n(M,\Lambda)\ge n$.
\end{tthm}

This theorem enables us to achieve our main purpose stated above; applying it to commutative rings, we can prove the following theorem, which is the second main result of this paper.
This extends Theorem \ref{egai} to arbitrary commutative noetherian local rings.
\begin{tthm}\label{im2}
Let $R$ be a commutative noetherian local ring satisfying $(\s_n)$.
The following are equivalent.
\begin{enumerate}[\rm(1)]
\item
$\syz^n(\mod R)$ is extension closed.
\item
$\syz^n(\mod R)=\s_n(R)$.
\item
$R_\p$ is Gorenstein for all $\p\in\spec R$ with $\height\p<n$.
\end{enumerate}
\end{tthm}

The organization of this paper is as follows.
The next Section 2 is devoted to introducing several notions and basic properties that are necessary in the later sections.
In Section 3, we consider extension closedness of syzygies.
We prove Theorem \ref{im1}, and apply it to commutative rings to get a sufficient condition for local Gorensteinness in height $n-1$.
In the final Section 4, we study local Gorensteinness of commutative rings.
For each integer $t\ge0$ we show the equivalence of local Gorensteinness in height equal to $t$ and local Gorensteinness in height at most $t$.
Combining this with a result obtained in Section 3, we finally give a proof of Theorem \ref{im2}.

\section{Preliminaries}

We start by our convention.

\begin{conv}
Throughout the rest of this paper, let $\Lambda$ be a two-sided noetherian ring, and let $R$ be a commutative noetherian ring.
We assume that all modules are finitely generated right ones, and that all subcategories are full ones.
\end{conv}

Denote by $\mod\Lambda$ the category of (finitely generated right) $\Lambda$-modules, by $\proj\Lambda$ the subcategory of projective modules.
Define the functor $(-)^*:\mod\Lambda\to\mod\Lambda^\op$ by $(-)^*=\Hom_\Lambda(-,\Lambda)$.
A subcategory $\X$ of $\mod\Lambda$ is said to be {\em extension closed} provided that for each exact sequence $0\to L\to M\to N\to0$ in $\mod\Lambda$, if $L$ and $N$ are in $\X$, then so is $M$.

We recall the definitions of minimal morphisms and approximations.

\begin{dfn}
Let $\C$ be a category, and let $\X$ be a subcategory of $\C$.
Let $\phi:C\to X$ be a morphism in $\C$ with $X\in\X$.
We say that:
\\
(1) $\phi$ is {\em left minimal} if all endomorphisms $f:X\to X$ with $\phi=f\phi$ are automorphisms.\\
(2) $\phi$ is a {\em left $\X$-approximation} if all morphisms from $C$ to objects in $\X$ factor through $\phi$.\\
A {\em right minimal} morphism and a {\em right $\X$-approximation} are defined dually.
\end{dfn}

A left (respectively, right) $\X$-approximation is sometimes called an {\em $\X$-preenvelope} (respectively, {\em $\X$-precover}), and a left (respectively, right) minimal one an {\em $\X$-envelope} (respectively, {\em $\X$-cover}).
A homomorphism $\phi:M\to P$ of $\Lambda$-modules with $P$ projective is a left $\proj\Lambda$-approximation if and only if the $\Lambda$-dual map $\phi^*:P^*\to M^*$ is surjective.
For the details of minimal morphisms and approximations, see \cite[\S1]{AR3} for instance.

Next we recall the definitions of an adjoint pair, a unit and a counit.

\begin{dfn}
Let $\L:\C\to\D$ and $\R:\D\to\C$ be functors between categories $\C$ and $\D$.
Suppose that for all $X\in\C$ and $Y\in\D$ there is a functorial isomorphism
$$
\Phi_{XY}:\Hom_\D(\L X,Y)\xrightarrow{\cong}\Hom_\C(X,\R Y).
$$
Then we say that $(\L,\R):\C\to\D$ (or more precisely, $(\L,\R,\Phi):\C\to\D$) is an {\em adjoint pair}.
Taking $\Phi_{X,\L X}(1_{\L X})$ and $\Phi_{\R Y,Y}^{-1}(1_{\R Y})$, one obtains natural transformations
$$
u:\one_\C\to\R\L,\qquad c:\L\R\to\one_\D,
$$
which are called the {\em unit} and {\em counit} of the adjunction, respectively.
\end{dfn}

For each $X\in\C$, every morphism $X\to\R Y$ with $Y\in\D$ uniquely factors through $uX:X\to\R\L X$.
Dually, for each $Y\in\D$, every morphism $\L X\to Y$ with $X\in\C$ uniquely factors through $cY:\L\R Y\to Y$.
In paricular, $uX$ and $cY$ are a left $\Im(\R)$-approximation and a right $\Im(\L)$-approximation, respectively.
(Here, for a functor $F:\X\to\Y$ we denote by $\Im(F)$ the {\em essential image} of $F$, namely, the subcategory of $\Y$ consisting of objects $N$ such that $N\cong FM$ for some $M\in\X$.)
Also, the equalities $\Phi_{XY}(f)=\R f\cdot uX$ and $\Phi_{XY}^{-1}(g)=cY\cdot\L g$ hold for all morphisms $f:\L X\to Y$ and $g:X\to\R Y$.
Furthermore, the compositions of natural transformations $\L\xrightarrow{\L u}\L\R\L\xrightarrow{c\L}\L$ and $\R\xrightarrow{u\R}\R\L\R\xrightarrow{\R c}\R$ are both identities.
The details can be found in \cite[Theorem IV.1.1]{M}.

Let us recall the definitions of syzygies, transposes and stable categories.

\begin{dfn}
(1) Let $M$ be a $\Lambda$-module.
Let $\cdots\to P_n\xrightarrow{\partial_n}P_{n-1}\to\cdots\to P_1\xrightarrow{\partial_1}P_0\to M\to0$ be a projective resolution of $M$.

(a) The {\em $n$th syzygy} $\syz^nM$ of $M$ is defined as the image of the map $\partial_n:P_n\to P_{n-1}$.

(b) The {\em transpose} $\tr M$ of $M$ is defined to be the cokernel of the map $\partial_1^*:P_0^*\to P_1^*$.\\
(2) For $\Lambda$-modules $M$ and $N$, let $\lhom_\Lambda(M,N)$ be the quotient of $\Hom_\Lambda(M,N)$ by the $\Lambda$-homomorphisms $M\to N$ factoring through some projective $\Lambda$-modules.
The residue class in $\lhom_\Lambda(M,N)$ of an element $f\in\Hom_\Lambda(M,N)$ is denoted by $\underline f$.\\
(3) We denote by $\lmod\Lambda$ the {\em stable category} of $\mod\Lambda$.
The objects of $\lmod\Lambda$ are precisely the (finitely generated right) $\Lambda$-modules, and the hom-set from a $\Lambda$-module $M$ to a $\Lambda$-module $N$ is given by $\lhom_\Lambda(M,N)$.
\end{dfn}

The modules $\syz M$ and $\tr M$ are uniquely determined by $M$ up to projective summands.
The assignments $M\mapsto\syz M$ and $M\mapsto\tr M$ give rise to additive functors
$$
\syz:\lmod\Lambda\to\lmod\Lambda,\qquad\tr:\lmod\Lambda\to\lmod\Lambda^\op.
$$
Moreover, $\tr$ is a duality, i.e., one has $\tr\tr M\cong M$ in $\lmod\Lambda$ for each $M\in\lmod\Lambda$.
We define the functor $\d_n:\lmod\Lambda\to\lmod\Lambda^\op$ by $\d_n=\syz^n\tr$; one has $\d_2M\cong M^*$ in $\lmod\Lambda^\op$ for each $M\in\lmod\Lambda$.
We denote by $\syz^n(\mod\Lambda)$ the subcategory of $\mod\Lambda$ consisting of $n$th syzygies $X$, that is, modules $X$ admitting an exact sequence $0 \to X \to P_{n-1} \to \cdots \to P_1 \to P_0$ of $\Lambda$-modules with each $P_i$ projective.
Also, $\syz^n(\lmod\Lambda)$ denotes the essential image of the functor $\syz^n:\lmod\Lambda\to\lmod\Lambda$, which coincides with the essential image of $\syz^n(\mod\Lambda)$ by the canonical functor $\mod\Lambda\to\lmod\Lambda$.
We refer the reader to \cite[\S2]{AB} for the details of syzygies, transposes and stable categories.

Finally, we recall the definitions of grade, depth and Serre's condition.

\begin{dfn}
(1) The {\em grade} of a $\Lambda$-module $M$, denoted by $\grade_\Lambda M$, is defined to be the infimum of nonnegative integers $i$ such that $\Ext_\Lambda^i(M,\Lambda)\ne0$.\\
(2) The {\em grade} of an ideal $I$ of $R$ is defined by $\grade I=\grade_R(R/I)$.\\
(3) When $R$ is a local ring with maximal ideal $\m$, the {\em depth} of an $R$-module $M$, denoted by $\depth_RM$, is defined as the infimum of nonnegative integers $i$ with $\Ext_R^i(R/\m,M)\ne0$.\\
(4) Let $n$ be a nonnegative integer.
An $R$-module $M$ is said to satisfy {\em Serre's condition $(\s_n)$} if the inequality $\depth_{R_\p}M_\p\ge\inf\{n,\height\p\}$ holds for all prime ideals $\p$ of $R$.
\end{dfn}

The maximal regular sequences on $R$ in $I$ (respectively, on $M$ in $\m$) have the same length, and this common length is equal to $\grade I$ (respectively, $\depth_RM$).
If $R$ satisfies $(\s_n)$, then $\height\p=\grade\p$ for all prime ideals $\p$ with $\height\p\le n$.
We denote by $\s_n(R)$ the subcategory of $\mod R$ consisting of modules satisfying $(\s_n)$.
This is an extension closed subcategory.
See \cite[\S1 and \S2]{BH} for the details of grade, depth and Serre's condition.

\section{Extension closedness of syzygies}

Let $X,Y$ be $\Lambda$-modules.
Let $\cdots \to P_1 \to P_0 \to \tr X\to 0$ and $\cdots\to Q_1\to Q_0\to Y\to0$ be projective resolutions.
Let $f:\tr\syz^n\tr X\to Y$ be a homomorphism of $\Lambda$-modules.
We extend $f$ to a chain map of complexes as in the left below, and make a commutative diagram with exact rows as in the right below.
$$
\xymatrix@R-1pc@C-1pc{
P_0^*\ar[r]\ar[d]^{f_{n+1}} & P_1^*\ar[r]\ar[d]^{f_n} & \cdots\ar[r] & P_{n+1}^*\ar[r]\ar[d]^{f_0} & \tr\syz^n\tr X\ar[r]\ar[d]^f & 0 \\
Q_{n+1}\ar[r] & Q_n\ar[r] & \cdots\ar[r] & Q_0\ar[r] & Y\ar[r] & 0
}
\qquad
\xymatrix@R-1pc@C-1pc{
P_0^*\ar[r]\ar[d]^{f_{n+1}} & P_1^*\ar[r]\ar[d]^{f_n} & X\ar[r]\ar[d]^{\phi(f)} & 0\\
Q_{n+1}\ar[r] & Q_n\ar[r] & \syz^nY\ar[r] & 0
}
$$
Thus we get a homomorphism $\phi(f):X\to\syz^nY$ of $\Lambda$-modules.
Conversely, let $g:X\to\syz^nY$ be a homomorphism of $\Lambda$-modules.
First, extend $g$ to a commutative diagram with exact rows as in the left below, whose rows are finite projective presentations.
Second, extend $h_0:=g_1^*$ and $h_1:=g_0^*$ to a chain map as in the middle below.
Third, make a commutative diagram with exact rows as in the right below.
$$
\xymatrix@R-1pc@C-1pc{
P_0^*\ar[r]\ar[d]^{g_1} & P_1^*\ar[r]\ar[d]^{g_0} & X\ar[r]\ar[d]^g & 0\\
Q_{n+1}\ar[r] & Q_n\ar[r] & \syz^nY\ar[r] & 0
}
\xymatrix@R-1pc@C-1pc{
Q_0^*\ar[r]\ar[d]^{h_{n+1}} & Q_1^*\ar[r]\ar[d]^{h_n} & \cdots\ar[r] & Q_n^*\ar[r]\ar[d]^{h_1} & Q_{n+1}^*\ar[d]^{h_0}\\
P_{n+1}\ar[r] & P_n\ar[r] & \cdots\ar[r] & P_1\ar[r] & P_0
}
\xymatrix@R-1pc@C-1pc{
P_n^*\ar[r]\ar[d]^{h_n^*} & P_{n+1}^*\ar[r]\ar[d]^{h_{n+1}^*} & \tr\syz^n\tr X\ar[r]\ar[d]^{\psi(g)} & 0\\
Q_1\ar[r] & Q_0\ar[r] & Y\ar[r] & 0
}
$$
Thus we get a homomorphism $\psi(g):\tr\syz^n\tr X\to Y$ of $\Lambda$-modules.
The following lemma holds; see \cite[Corollary 3.3]{AR}.

\begin{lem}\label{adj}
With the notation above, the assignments $\Phi:\underline{f}\mapsto\underline{\phi(f)}$ and $\Psi:\underline{g}\mapsto\underline{\psi(g)}$ make functorial isomorphisms
$$
\Phi:\lhom_\Lambda(\tr\syz^n\tr X,Y)\rightleftarrows\lhom_\Lambda(X,\syz^nY):\Psi
$$
which are mutually inverse.
In particular, one has an adjoint pair $(\tr\syz^n\tr,\syz^n):\lmod\Lambda\to\lmod\Lambda$.
(The corresponding statement for $\Lambda^\op$ also holds.)
\end{lem}

A {\em noether algebra} is by definition a module-finite algebra of a commutative noetherian ring.
(Thus a noether algebra is a two-sided noetherian ring.)
Now we prove the following theorem, which is the main result of this section.

\begin{thm}[\,=\,Theorem \ref{im1}]\label{main1}
Let $\Lambda$ be a noether algebra, $M$ a module over $\Lambda^\op$, and $n$ a nonnegative integer.
Suppose that $\syz^n(\mod\Lambda)$ is an extension closed subcategory of $\mod\Lambda$.
If $\grade_\Lambda\Ext_{\Lambda^\op}^i(M,\Lambda)\ge i-1$ for all $1\le i\le n$, then $\grade_\Lambda\Ext_{\Lambda^\op}^n(M,\Lambda)\ge n$.
\end{thm}

\begin{proof}
The assertion is trivial for $n=0$, and follows from Theorem \ref{ar} for $n=1$.
So assume $n\ge2$.
We use the notation of the part preceding Lemma \ref{adj}.
Set $X:=\tr M$, $Y:=\tr\syz^n\tr X$, $Q_0:=P_{n+1}^*$, $Q_1:=P_n^*$ and let $f$ be the identity map of $Z:=\tr\syz^n\tr X$.
We get a chain map as in the left below and a commutative diagram as in the right below.
$$
\xymatrix@R-1pc@C-1pc{
P_0^*\ar[r]^\alpha\ar[d]^{f_{n+1}} & P_1^*\ar[r]^\beta\ar[d]^{f_n} & \cdots\ar[r] & P_{n-1}^*\ar[r]\ar[d]^{f_2} & P_n^*\ar[r]\ar@{=}[d]^{f_1} & P_{n+1}^*\ar[r]\ar@{=}[d]^{f_0} & Z\ar[r]\ar@{=}[d]^f & 0 \\
Q_{n+1}\ar[r] & Q_n\ar[r] & \cdots\ar[r] & Q_2\ar[r] & Q_1\ar[r] & Q_0\ar[r] & Z\ar[r] & 0
}
\qquad
\xymatrix@R-1pc@C-1pc{
P_0^*\ar[r]^\alpha\ar[d]^{f_{n+1}} & P_1^*\ar[r]^\pi\ar[d]^{f_n} & X\ar[r]\ar[d]^{\phi(f)} & 0\\
Q_{n+1}\ar[r] & Q_n\ar[r] & \syz^nZ\ar[r] & 0
}
$$
These diagrams induce a chain map
$$
\xymatrix@R-1pc@C-1pc{
A\ar@<-2pt>[d]^\xi:& 0\ar[r] & X\ar[r]^\eta\ar[d]^\rho & P_2^*\ar[r]\ar[d]^{f_{n-1}} & P_3^*\ar[r]\ar[d]^{f_{n-2}} & \cdots\ar[r] & P_{n-2}^*\ar[r]\ar[d]^{f_3} & P_{n-1}^*\ar[r]\ar[d]^{f_2} & W\ar[r]\ar@{=}[d] & 0 \\
B:& {\underbrace{0}_{0}}\ar[r] & {\underbrace{V}_{1}}\ar[r] & {\underbrace{Q_{n-1}}_{2}}\ar[r] & {\underbrace{Q_{n-2}}_{3}}\ar[r] & \cdots\ar[r] & {\underbrace{Q_3}_{n-2}}\ar[r] & {\underbrace{Q_2}_{n-1}}\ar[r] & {\underbrace{W}_{n}}\ar[r] & {\underbrace{0}_{n+1}}
}
$$
where we put $\rho:=\phi(f)$, $V:=\syz^nZ=\syz^n\tr\syz^n\tr X$ and $W:=\syz^2\tr\syz^n\tr X$.
Note that for each $1\le i\le n$ the $i$th homology $\H^i(A)$ of the cochain complex $A$ is isomorphic to $\Ext_{\Lambda^\op}^i(\tr X,\Lambda)$, while $B$ is an exact complex.
Take the mapping cone of the chain map $\xi$.
It is easy to see that it is quasi-isomorphic to a complex
$$
\xymatrix@R-1pc@C-1pc{
C: & {\underbrace{0}_{-1}}\ar[r] & {\underbrace{X}_0}\ar[r]^(0.4){\binom{\rho}{\eta}} & {\underbrace{V\oplus P_2^*}_1}\ar[r] & {\underbrace{Q_{n-1}\oplus P_3^*}_2}\ar[r] & \cdots\ar[r] & {\underbrace{Q_3\oplus P_{n-1}^*}_{n-2}}\ar[r] & {\underbrace{Q_2}_{n-1}}\ar[r] & {\underbrace{0}_n}
}
$$
with $\H^i(C)\cong\Ext_{\Lambda^\op}^{i+1}(\tr X,\Lambda)\cong\Ext_{\Lambda^\op}^{i+1}(M,\Lambda)$ for $0\le i\le n-1$.

By Lemma \ref{adj} we have an adjoint pair $(\L,\R):\lmod\Lambda\to\lmod\Lambda$ with $\L=\tr\syz^n\tr$ and $\R=\syz^n$.
Let $u:\one\to\R\L$ be the unit of the adjunction.
Then the morphism $\underline\rho$ is nothing but $uX$.
We establish several claims.

\begin{claim}\label{c1}
The morphism $\binom{\rho}{\eta}:X\to V\oplus P_2^*$ in $\mod\Lambda$ is a left $\syz^n(\mod\Lambda)$-approximation.
\end{claim}

\begin{proof}[Proof of Claim]
As $\underline\rho=uX$, the morphism $\underline\rho$ in $\lmod\Lambda$ is a left $\syz^n(\lmod\Lambda)$-approximation.
We have $\beta=\eta\pi$, and the complex $P_2^{**}\xrightarrow{\beta^*}P_1^{**}\xrightarrow{\alpha^*}P_0^{**}$ is exact, since it is isomorphic to the complex $P_2\to P_1\to P_0$.
It is observed from this that the map $\eta^*:P_2^{**}\to X^*$ is surjective, which implies that the morphism $\eta$ in $\mod\Lambda$ is a left $\proj\Lambda$-approximation.
The claim can now easily be shown.
\renewcommand{\qedsymbol}{$\square$}
\end{proof}

\begin{claim}\label{tf}
If $\Ext_{\Lambda^\op}^i(\tr X,\Lambda)=0$ for all $1\le i\le n$ (i.e., $X$ is {\em $n$-torsionfree}), then $uX$ is an isomorphism in $\lmod\Lambda$.
(The corresponding statement for $\Lambda^\op$ is also true.)
\end{claim}

\begin{proof}[Proof of Claim]
The assumption implies that the complex $P_0^*\to P_1^*\to\cdots\to P_{n+1}^*\to\tr\syz^n\tr X\to 0$ is exact.
Hence we can take $Q_i:=P_{n+1-i}^*$ and $f_i:=1_{P_{n+1-i}^*}$ for every $0\le i\le n+1$.
It follows that $uX$ is an isomorphism.
\renewcommand{\qedsymbol}{$\square$}
\end{proof}

\begin{claim}\label{ddd}
\text{The morphism $u\d_nX:\d_nX\to\d_n^3X$ in $\lmod\Lambda$ is an isomorphism}.
\end{claim}

\begin{proof}[Proof of Claim]
Combining the assumption of the theorem with \cite[Proposition (2.26)]{AB} yields that $\d_nX=\syz^nM$ is $n$-torsionfree.
The assertion follows from Claim \ref{tf}.
\renewcommand{\qedsymbol}{$\square$}
\end{proof}

\begin{claim}\label{duud}
The equality $\underline{1_{\d_nZ}}=\d_nuZ\cdot u\d_nZ$ holds for all $Z\in\lmod\Lambda$.
(The corresponding statement for $\Lambda^\op$ is also true.)
\end{claim}

\begin{proof}[Proof of Claim]
Let $c:\L\R\to\one$ stand for the counit of the adjunction.
The composition $\R c\cdot u\R$ of natural transformations is an identity.
It is straightforward to verify that $c$ coincides with the composition $\tr u\tr$.
The assertion is now easily deduced.
\renewcommand{\qedsymbol}{$\square$}
\end{proof}

\begin{claim}\label{c2}
The morphism $\underline\rho:X\to V$ in $\lmod\Lambda$ is left minimal.
\end{claim}

\begin{proof}[Proof of Claim]
We have $\underline\rho=uX$ and $V=\d_n^2X$.
Let us prove that $uX:X\to\d_n^2X$ is left minimal.
Let $h:\d_n^2X\to\d_n^2X$ be an endomorphism in $\lmod\Lambda$ such that $uX=h\cdot uX$.
Then $\d_nuX=\d_nuX\cdot\d_nh$.
By Claims \ref{ddd} and \ref{duud} the morphism $\d_nuX$ is an isomorphism.
Hence $\d_nh$ is an automorphism, and so is $\d_n^2h$.
Using Claim \ref{duud} again, we have $\underline{1_{\d_n^2X}}=\d_nu\d_nX\cdot u\d_n^2X$.
Applying Claim \ref{ddd} again implies that $\d_nu\d_nX$ is an isomorphism.
Hence $u\d_n^2X$ is also an isomorphism.
It follows from the equality $\d_n^2h\cdot u\d_n^2X=u\d_n^2X\cdot h$ that $h$ is an automorphism.
Therefore, the morphism $uX$ is left minimal.
\renewcommand{\qedsymbol}{$\square$}
\end{proof}
Let $E$ be the cokernel of the map $\binom{\rho}{\eta}$.
Recall our assumption that $\Lambda$ is a noether algebra and $\syz^n(\mod\Lambda)$ is extension closed.
Combining Claims \ref{c1} and \ref{c2} with \cite[Corollary 1.8]{AR2} and \cite[Lemma 4.5]{AR}, we obtain $\Ext_\Lambda^1(E,\syz^n(\mod\Lambda))=0$.
In particular, $\Ext_\Lambda^1(E,\Lambda)=0$.

Decomposing the complex $C$ into short exact sequences
\begin{equation}\label{two}
0\to B^i\xrightarrow{p^i} Z^i\to H^i\to0,\quad0\to Z^i\xrightarrow{q^i} C^i\to B^{i+1}\to0\quad(0\le i\le n-1)
\end{equation}
with $H^i=\H^i(C)\cong\Ext_{\Lambda^\op}^{i+1}(M,\Lambda)$.
The assumption of the theorem implies $\grade_\Lambda H^i\ge i$ for all $0\le i\le n-1$.

What we want to prove is that $\grade_\Lambda H^{n-1}\ge n$.
Making the pushout diagram of the maps $p^1$ and $q^1$, we get an exact sequence $0 \to H^1 \to E \to B^2 \to 0$.
When $n=2$, we have $B^2=0$ and $H^1=E$.
Hence $\Ext_\Lambda^1(H^1,\Lambda)=0$, which implies $\grade_\Lambda H^1\ge2$ and we are done.
Let $n\ge3$.
The functor $(-)^*$ gives an exact sequence $0=(H^1)^*\to\Ext_\Lambda^1(B^2,\Lambda)\to\Ext_\Lambda^1(E,\Lambda)=0$, which shows $\Ext_\Lambda^1(B^2,\Lambda)=0$.
Let $2\le i\le n-2$ be an integer.
From \eqref{two} we get an exact sequence $0=\Ext_\Lambda^{i-1}(H^i,\Lambda)\to\Ext_\Lambda^{i-1}(Z^i,\Lambda)\to\Ext_\Lambda^{i-1}(B^i,\Lambda)$ and an isomorphism $\Ext_\Lambda^{i-1}(Z^i,\Lambda)\cong\Ext_\Lambda^i(B^{i+1},\Lambda)$ since $C^i$ is a projective module.
Therefore we have an injection $\Ext_\Lambda^i(B^{i+1},\Lambda)\hookrightarrow\Ext_\Lambda^{i-1}(B^i,\Lambda)$.
Thus we obtain a chain
$$
\Ext_\Lambda^{n-2}(B^{n-1},\Lambda)\hookrightarrow\cdots\hookrightarrow\Ext^2(B^3,\Lambda)\hookrightarrow\Ext_\Lambda^1(B^2,\Lambda)
$$
of injections.
Since $\Ext_\Lambda^1(B^2,\Lambda)$ vanishes, so does $\Ext_\Lambda^{n-2}(B^{n-1},\Lambda)$.
The exact sequence $0\to B^{n-1}\to C^{n-1}\to H^{n-1}\to0$ and the projectivity of the module $C^{n-1}$ imply $\Ext_\Lambda^{n-1}(H^{n-1},\Lambda)=0$.
Now we conclude $\grade_\Lambda H^{n-1}\ge n$.
\end{proof}

\begin{rem}
Theorem \ref{main1} is regarded as a strong version of the implication (1) $\Rightarrow$ (2) in Theorem \ref{ar}.
In fact, one can deduce this implication from Theorem \ref{main1}, as follows.
We use induction on $n$; the case $n=0$ is trivial.
Let $n\ge1$, and assume that $\syz^i(\mod\Lambda)$ is extension closed for all $1\le i\le n$.
The induction hypothesis yields $\grade_\Lambda\Ext_{\Lambda^\op}^i(M,\Lambda)\ge i$ for each $1\le i\le n-1$ and each $M\in\mod\Lambda^\op$.
We have $\grade_\Lambda\Ext_{\Lambda^\op}^n(M,\Lambda)\ge n-1$: this is trivial for $n=1$, and for $n\ge2$ the isomorphism $\Ext_{\Lambda^\op}^n(M,\Lambda)\cong\Ext_{\Lambda^\op}^{n-1}(\syz M,\Lambda)$ implies it.
By virtue of Theorem \ref{main1}, we obtain $\grade_\Lambda\Ext_{\Lambda^\op}^n(M,\Lambda)\ge n$.
Thus $\grade_\Lambda\Ext_{\Lambda^\op}^i(M,\Lambda)\ge i$ for all $1\le i\le n$ and all $M\in\mod\Lambda^\op$.
\end{rem}

The following result is a consequence of Theorem \ref{main1}, which is used in the proof of Theorem \ref{main2} stated later.

\begin{cor}\label{mcor}
Let $n$ be a nonnegative integer.
Let $R$ be a commutative noetherian ring satisfying Serre's condition $(\s_{n-1})$.
If $\syz^n(\mod R)$ is extension closed, then the local ring $R_\p$ is Gorenstein for all prime ideals $\p$ of $R$ with height $n-1$.
\end{cor}

\begin{proof}
Let $\p$ be a prime ideal of $R$ with $\height\p=n-1$.
We have
$$
\grade_R\Ext_R^i(R/\p,R)=\grade(\ann_R\Ext_R^i(R/\p,R))\ge\grade\p=n-1\ge i-1
$$
for each $1\le i\le n$.
Here, the first inequality is shown by the fact that the ideal $\ann_R\Ext_R^i(R/\p,R)$ contains $\p$, and the second equality follows from the assumption that $R$ satisfies $(\s_{n-1})$.
Applying Theorem \ref{main1}, we obtain $\grade_R\Ext_R^n(R/\p,R)\ge n$, and therefore $\Ext_R^{n-1}(\Ext_R^n(R/\p,R),R)=0$.
Localization at $\p$ yields $\Ext_{R_\p}^{n-1}(\Ext_{R_\p}^n(\kappa(\p),R_\p),R_\p)=0$.
Suppose that $\Ext_{R_\p}^n(\kappa(\p),R_\p)$ is nonzero.
Then it contains $\kappa(\p)$ as a direct summand, and $\Ext_{R_\p}^{n-1}(\kappa(\p),R_\p)$ is a direct summand of $\Ext_{R_\p}^{n-1}(\Ext_{R_\p}^n(\kappa(\p),R_\p),R_\p)$, which is zero.
If follows that $\Ext_{R_\p}^{n-1}(\kappa(\p),R_\p)=0$, but this contradicts the fact that $R_\p$ has depth $n-1$.
(In fact, $R_\p$ is a Cohen--Macaulay local ring of dimension $n-1$.)
Therefore $\Ext_{R_\p}^n(\kappa(\p),R_\p)=0$, which implies that $R_\p$ is Gorenstein; see \cite[Theorem (1.1)]{FFGR}.
\end{proof}

We close this section by posing a naive question.

\begin{ques}
Under the assumption of Corollary \ref{mcor}, is $R_\p$ a Gorenstein local ring for all prime ideals $\p$ with height {\em less than} (or equal to) $n-1$\,?
\end{ques}

\section{Local Gorensteinness of commutative rings}

We begin with proving the following theorem.

\begin{thm}\label{rigid}
Let $n>t>0$ be integers.
Let $R$ be a commutative noetherian local ring of dimension $d\ge t$ satisfying Serre's condition $(\s_n)$.
Then the following are equivalent.
\begin{enumerate}[\rm(1)]
\item
$R_\p$ is Gorenstein for all $\p\in\spec R$ with $\height\p=t$.
\item
$R_\p$ is Gorenstein for all $\p\in\spec R$ with $\height\p\le t$.
\end{enumerate}
\end{thm}

\begin{proof}
It suffices to prove that (1) implies (2).
It is enough to show that $R_\p$ is Gorenstein for all $\p\in\spec R$ with $\height\p=t-1$.
Suppose that this statement is not true, and consider the counterexample where $d=\dim R$ is minimal.
There exists a prime ideal $\q$ with height $t-1$ such that $R_\q$ is not Gorenstein.
If $t>1$, then $\grade\q=t-1>0$ and there is an $R$-regular element $x$ in $\q$.
We have $n-1>t-1>0$ and $d-1\ge t-1$.
The residue ring $R/xR$ is a $(d-1)$-dimensional local ring that satisfies $(\s_{n-1})$ and is locally Gorenstein in height $t-1$.
The prime ideal $\q/xR$ of $R/xR$ has height $(t-1)-1$, and $(R/xR)_{\q/xR}$ is non-Gorenstein.
This contradicts the minimality of $d$, and we must have $t=1$.

Let $0=\bigcap_{\p\in\ass R}L(\p)$ be a primary decomposition of the zero ideal $0$ of $R$.
Let $A$ be the set of associated primes $\p$ such that $R_\p$ is not Gorenstein, and set $B=\ass R\setminus A$.
As $\q$ has height $t-1=0$, it belongs to $A$.
Take a prime ideal $\r$ of height $t$.
The assumption (1) shows that $R_\r$ is Gorenstein.
Choosing a minimal prime $\ss$ contained in $\r$, we see that $\ss$ belongs to $B$.
Thus both $A$ and $B$ are nonempty.
Put $I=\bigcap_{\p\in A}L(\p)$ and $J=\bigcap_{\p\in B}L(\p)$.

We claim that the ideal $I+J$ is $\m$-primary, where $\m$ stands for the maximal ideal of $R$.
Indeed, assume that there exists a nonmaximal prime ideal $P$ containing $I+J$.
Then $P$ contains some prime ideals $P_1\in A$ and $P_2\in B$.
Set $e:=\dim R_P$.
We have $e<d$, and the fact that $P_1\ne P_2$ implies $e\ge1=t$.
The local ring $R_P$ satisfies $(\s_n)$ and is locally Gorenstein in height $t$.
Also, $(R_P)_{P_1R_P}=R_{P_1}$ is non-Gorenstein, and $\height P_1R_P=\height P_1=0=t-1$, since $R$ satisfies $(\s_1)$ and the equality $\ass R=\Min R$ holds.
We thus get a contradiction to the minimality of $d$, and the claim follows.

Let $X=\spec R\setminus\{\m\}$ be the punctured spectrum, and let $V_1=\v(I)\cap X$ and $V_2=\v(J)\cap X$ be closed subsets of $X$.
For each $i=1,2$ the set $V_i$ is nonempty since it contains $P_i$, while $V_1\cap V_2$ is empty.
Thus $X$ is disconnected, and Hartshorne's connectedness theorem \cite[Proposition 2.1]{H} implies that $R$ has depth at most $1$.
Since $R$ satisfies $(\s_1)$ and $d\ge t>0$, it is a Cohen--Macaulay local ring of dimension $1$.
Our assumption (1) implies that $R=R_\m$ is Gorenstein, and so is $R_\q$, which is a contradiction.
\end{proof}

As is seen in the following example, the conclusion of Theorem \ref{rigid} does not necessarily hold if one removes the assumption that $R$ is local.

\begin{ex}
Let $A$ and $B$ be commutative noetherian local rings with $\dim A\ge1$ and $\dim B=0$.
Assume that $A$ is locally Gorenstein in height one and that $B$ is non-Gorenstein.
Let $R=A\times B$ be a product ring.
Then $R$ is locally Gorenstein in height one, but not so in height zero.

Indeed, let $\p$ be a prime ideal of $R$ with height one.
Then $\p=P\times B$ for some prime ideal $P$ of $A$ with height one.
Hence $R_\p=A_P$ is Gorenstein.
Set $\q=A\times Q$, where $Q$ is the maximal ideal of $B$.
Then $\q$ is a minimal prime of $R$, and $R_\q=B$ is not Gorenstein.
\end{ex}

The next example says that the assertion of Theorem \ref{rigid} is not necessarily true if $n=t$.

\begin{ex}
Let $S$ be a regular local ring of dimension $3$ with regular system of parameters $x,y,z$.
Then $R=S/(x)\cap(y,z)^2$ is a local ring of dimension $2$ satisfying $(\s_1)$.
The ring $R$ is locally Gorenstein in height $1$, but not so in height $0$.

In fact, set $\p=xS$ and $\q=(y,z)S$.
We have $\ass R=\Min R=\{\p R,\q R\}\supsetneq\{\p R\}=\assh R$.
Hence $R$ has dimension $2$, depth $1$ and satisfies $(\s_1)$.
Let $P$ be a prime ideal of $R$ with height $1$.
Write $P=Q/\p\cap\q^2$ for some prime ideal $Q$ of $S$ containing $\p\cap\q^2$.
If $Q$ contains $\q$, then $(y,z)S\subsetneq Q\subsetneq(x,y,z)S$, which gives a contradiction.
Thus $Q$ does not contain $\q$ but contains $\p$, and the ring $R_P=S_Q/xS_Q$ is regular, whence Gorenstein.
On the other hand, $R_{\q R}=S_\q/\q^2S_\q$ is an artinian local ring of type $2$, whence non-Gorenstein.
\end{ex}

Now we prove the following theorem, which is the main result of this section.
We should remark that this theorem extends Theorem \ref{egai} on Cohen--Macaulay local rings with canonical module to arbitrary commutative noetherian local rings.
Compare with Corollary \ref{mcor} the implication (1) $\Rightarrow$ (3) in this theorem.

\begin{thm}[\,=\,Theorem \ref{im2}]\label{main2}
Let $n$ be a nonnegative integer.
Let $R$ be a commutative noetherian local ring satisfying Serre's condition $(\s_n)$.
The following are equivalent.
\begin{enumerate}[\rm(1)]
\item
The subcategory $\syz^n(\mod R)$ of $\mod R$ is extension closed.
\item
The equality $\syz^n(\mod R)=\s_n(R)$ holds.
\item
The local ring $R_\p$ is Gorenstein for all prime ideals $\p$ of $R$ with height less than $n$.
\end{enumerate}
\end{thm}

\begin{proof}
It is shown in \cite[Theorem 3.8]{EG} (see also \cite[Lemma 1.3]{LW}) that (3) implies (2).
It is straightforward to check that (2) implies (1).
Let us show that (1) implies (3).
Thanks to Corollary \ref{mcor}, $R$ is locally Gorenstein in height $n-1$.
We may assume $n>1$.
Applying Theorem \ref{rigid} to $t:=n-1$ yields that $R$ is locally Gorenstein in height at most $n-1$.
\end{proof}

\begin{ac}
The second author thanks Tokuji Araya for valuable discussions.
\end{ac}

\end{document}